\documentclass{amsart}
\usepackage{amsmath,amsthm,amssymb}
\usepackage{lipsum} 
\usepackage[utf8]{inputenc}
\usepackage{esint}
\usepackage[all]{xy}
\usepackage{amsfonts} 
\usepackage{xcolor}
\usepackage{tikz}

\usepackage{color}
\usepackage[colorlinks=true]{hyperref}
\hypersetup{
    linkcolor=cyan,          
    citecolor=orange,        
    filecolor=cyan,      
    urlcolor=cyan
}

\usepackage{bbold}
\usepackage{enumerate}
\usepackage{calc}
\usepackage{bbold}
\usepackage[colorinlistoftodos,prependcaption,textsize=tiny]{todonotes}
\usepackage{todonotes}
\usepackage[shortcuts]{extdash}
\usepackage[style=alphabetic,maxnames=200,maxalphanames=6,giveninits]{biblatex}
\makeatletter
\allowdisplaybreaks
\numberwithin{equation}{section}
\theoremstyle{plain}
        \newtheorem{theorem}{Theorem}[section]

        \newtheorem{proposition}[theorem]{Proposition}
        \newtheorem{lemma}[theorem]{Lemma}
         
        \newtheorem{definition}[theorem]{Definition} 
        \newtheorem{remark}[theorem]{Remark}

\newtheorem*{theorem*}{Theorem}
\newtheorem*{definition*}{Definition}
\newtheorem*{proposition*}{Proposition}

\def\R{\mathbb{R}}

\def\cP{\mathbb{P}}

\def\M{\mathcal{M}}

\def\cD{\mathcal{D}}

\def\cH{\mathcal{H}}
\def\cP{\mathcal{P}}
\def\cU{\mathcal{U}}

\newcommand{\bbnabla}{%
    \,\nabla\mkern-12mu\nabla}
\newcommand{\TN}{%
   \widetilde {\bbnabla}}

\newcommand{\sd}{\mathbb{S}^{2}}

\def\af{\mathsf F}

\def\aU{\mathsf U}

\def\ar{\mathsf r}

\def\G{\R^{2d}}

\def\TGRE{\mathcal{GRE}}
\def\cL{\mathcal{L}}
\def\cJ{\mathcal{J}}
\def\cA{\mathcal{A}}
\def\cM{\mathcal{M}}

\def\supp{\operatorname{supp}}

\def\tn{\widetilde\nabla}

\newcommand{\defeq }{\mathop{=}\limits^{\textrm{def}}}
\def\d{\partial}
\def\Do{\R^{d}}
\newcommand{\dd}{{\,\rm d}}

\title{Gradient flow approach to Landau equation: A cross-product structure
}
\author[H.~Duong]{Manh Hong Duong}
\author[Z.~He]{Zihui He}

\address[H.~Duong]
{School of Mathematics, University of Birmingham, UK}
\email{h.duong@bham.ac.uk}
\address[Z.~He]
{Fakult\"at f\"ur Mathematik, Universit\"at Bielefeld, Postfach 100131, 33501 Bielefeld, Germany}
\email{zihui.he@uni-bielefeld.de}

\date{\today}

\begin{document}

\begin{abstract}
We introduce a cross-product Landau gradient that commutes with Gaussian mollification. For the hard-potential interaction kernels of the form
$A(|v-v_*|)\sim \langle v-v_*\rangle^\gamma|v-v_*|^2$ with $\gamma\in(-\infty,1]$,
this construction yields a variational characterisation of the gradient-flow structure of the spatially homogeneous Landau equation. The cross-product operator induces the same gradient-flow geometry as the Landau gradient in \cite{carrillo2024landau}, while providing a different representation of the underlying Landau gradient. The main advantage of this formulation is that it requires only weighted $L^1$-control. We also discuss a similar variational characterisation for the GENERIC structure of the fuzzy Landau equations.
\end{abstract}

\maketitle

\section{Introduction}
In this paper, we consider the homogeneous  Landau equation
\begin{equation}
    \label{Landau}
    \d_t f =Q(f,f)\quad\text{on}\quad[0,T]\times \Do.
\end{equation}
 This equation characterises the time evolution of the unknown $f_t(v):[0,T]\times\Do\to\R_+$, describing the distribution of particles in a plasma at time $t$ with velocity $v$. The Landau collision term is given by
\begin{equation*}
\begin{aligned}
 Q(f,f)=&\nabla_v \cdot\Big(\int_{\Do}A(|v-v_*|)\Pi_{(v-v_*)^\perp}\big(f_*\nabla_v f-f\nabla_{v_*} f_*\big)\dd v_*\Big),
\end{aligned}
\end{equation*}
where  $f_*=f(v_*)$ denotes the density of particle at velocity $v_*$. In the above, $\Pi_{z^\perp}$, $z\in\Do$, denotes the projection operator onto $z^\perp$
\begin{align*}
  \Pi_{z^\perp}=\operatorname{Id}-\frac{z\otimes z}{|z|^2}.
\end{align*}
For technical reasons, throughout this paper we restrict our analysis to interaction kernels of the form
\begin{equation}
\label{A:T3}
\begin{gathered}
A(|v-v_*|)=A_0(|v-v_*|)|v-v_*|^2,\\
A_0(|v-v_*|)\sim\langle v-v_*\rangle^\gamma,
\quad \gamma\in(-\infty,1].
    \end{gathered}
\end{equation}
These kernels have the same large-velocity behaviour as $|v-v_*|^{\gamma+2}$ but are regularised near $v\sim v_*$.
In the following, for notational simplicity, we restrict ourselves to $d=3$. See Remark~\ref{rmk:d-big} for the discussions for the case $d\ge 3$.

A gradient-flow approach to the homogeneous Landau equation was shown by \textcite{carrillo2024landau}.  The authors showed that the Landau equation \eqref{Landau} can be interpreted as the gradient flow of the Boltzmann entropy
$$\cH(f)=\int f\log f$$
with respect to a Wasserstein-type metric associated to the following Landau gradient 
\begin{equation*}
    \tn_{\Pi}f=\sqrt{A}\Pi_{(v-v_*)^\perp}(\nabla f-\nabla_* f_*).
\end{equation*}
Their result applies to the power-law kernels
\begin{equation}
\label{cla:A0}
A(|v-v_*|)=|v-v_*|^{\gamma+2},
\qquad \gamma\in(-3,0],
\end{equation}
under additional $L^p$ bounds along the curves, with $p>\frac{3}{3-|\gamma|}$.

Here, instead of treating $A(|v-v_*|)\Pi_{(v-v_*)^\perp}$ as a single object, we separate it into the radial factor $A_0(|v-v_*|)$ and the geometric factor $|v-v_*|^2\Pi_{(v-v_*)^\perp}$. Using the representation
\begin{align}
\label{cross}
|v-v_*|^2\Pi_{(v-v_*)^\perp}u
=(v-v_*)\times\big(u\times(v-v_*)\big),
\qquad u\in\R^3,
\end{align}
we define a new cross-product form of the Landau gradient
\begin{equation}
\label{def:nabla-L}
\widetilde\nabla f
=(v-v_*)\times\big(\nabla_vf-\nabla_{v_*}f_*\big).
\end{equation}
Extending the cross-Landau gradient $\tn$ to the tensorised density $f\otimes f_*$ and by centre-of-mass coordinates, we obtain
\begin{align*}
\tn(f\otimes f_*)
=\sigma\times\nabla_\sigma(f\otimes f_*),
\qquad
\sigma=\frac{v-v_*}{|v-v_*|}\in\sd.
\end{align*}
This representation yields the key commutation property with Gaussian mollification
\begin{align*}
   \big[\tn,M_\beta *\big] =0,\quad M_\beta(v)=(2\pi\beta)^{-3/2}\exp\left(-\frac{|v|^2}{2\beta}\right)\quad \beta\in(0,1)
\end{align*}
with a slight abuse of notation.
The commutation properties of this gradient allow us to remove the additional $L^p$-bounds required in \cite{carrillo2024landau}. To control the remaining radial factor $A_0$, we restrict our analysis to the regularised kernels \eqref{A:T3}. A detailed discussion of the commutation property is given in Section~\ref{sec:com}. The relation between the gradient-flow structure induced by the cross-product operator and that of \cite{carrillo2024landau} is discussed in Section~\ref{sec:gf}.

We define the corresponding divergence by duality through the integration-by-parts formula
\begin{equation*}
\int_{\R^6}\widetilde\nabla f\cdot G\dd v_*\dd v
=-\int_{\R^3}f\widetilde\nabla\cdot G\dd v.
\end{equation*}
For sufficiently smooth and decaying $G:\R^6\to\R^3$, we have
\begin{equation*}
\widetilde\nabla\cdot G
=-\nabla_v\cdot\int_{\R^3}(v-v_*)\times
\big(G(v,v_*)+G(v_*,v)\big)\dd v_*.
\end{equation*}

We note that the Landau equation \eqref{Landau} can be written be written as
\begin{equation}
\label{GF}
   \d_t f=\frac12 \tn\cdot\big( A_0ff_*\tn \dd\cH(f)\big),\quad \dd\cH(f)=\log f+1.   
\end{equation}

As in \cite{carrillo2024landau}, we characterise the gradient flow structure via an entropy-dissipation inequality. To state our main result and characterise the metric speed of the curve, we introduce the grazing rate equation associated with the Landau gradient defined in \eqref{def:nabla-L}
\begin{equation}
\label{intro:TGRE}
    \d_t f+\frac12 \widetilde \nabla \cdot U_t=0,
\end{equation}
where $U:[0,T]\times \R^6\to \R^3$ denotes the grazing rate. Notice that, when $U_t=- A_0 ff_* \widetilde\nabla \log f$, the equation \eqref{intro:TGRE} recovers the Landau equation \eqref{GF}. We also define the following curve action associated to \eqref{intro:TGRE}
\begin{align*}
\cA(f,U)=\frac12\int\frac{|U|^2}{ff_*A_0} \dd v_*\dd v.
\end{align*}

We show the following main result.
\begin{theorem}\label{thm:main}
Let the kernel $A_0$ satisfy \eqref{A:T3}. Let $(f_t,U_t)$ be solutions to the grazing rate equation \eqref{intro:TGRE}. 
We assume $(f_t)_{t\in[0,T]}$ is a curve of probability density on $\R^3$ such that $f\in L^\infty([0,T];L^1_{2+\gamma_+}(\R^3))$, initial entropy $|\cH(f_0)|<+\infty$, and $\cD(f)\in L^1([0,T])$. Then we have
\begin{equation}
\label{J}
\cJ_T(f,U):=\cH(f_T)-\cH(f_0) +\frac12\int_0^T\cD(f_t)\dd t+\frac12\int_0^T\cA(f_t,U_t)\dd t \ge 0.
\end{equation}
Moreover, we have $\cJ_T(f,U)=0$ if and only if $f$ is an $\cH$-solution of the Landau equation \eqref{Landau}.
In this case, we have the entropy identity
\begin{equation}
\label{H-thm}
    \cH(f_T)-\cH(f_0) +\int_0^T\cD(f_t)\dd t=0.
\end{equation}

\end{theorem}

In the above, $\gamma_+:=\max(0,\gamma)$, and the functional space $L^1_{a}(\Do)$ consists all functions such that $\langle v\rangle^a f\in L^1(\Do)$,
where $\langle v\rangle:=\sqrt{1+|v|^2}$ denotes the Japanese bracket. The definition of $\cH$-solutions is given as in  \cite{Vil98b}.

\begin{remark}[Higher dimensional cases]\label{rmk:d-big}

The result also holds in dimensions $d\ge3$ by replacing the cross product with the skew-symmetric matrix
\begin{align*}
(a\wedge b)_{ij}:=a_i b_j-a_j b_i,
\qquad a,b\in\R^d.
\end{align*}
The higher-dimensional analogue of \eqref{cross} is, componentwise,
\begin{align*}
|v-v_*|^2\big(\Pi_{(v-v_*)^\perp}u\big)_j
=\sum_{i=1}^d(v-v_*)_i
\big((v-v_*)\wedge u\big)_{ij},
\end{align*}
for $u\in\R^d$ and $j=1,\ldots,d$. Correspondingly, we define the matrix-valued cross-Landau gradient by
\begin{align*}
\tn f
=(v-v_*)\wedge\big(\nabla_vf-\nabla_{v_*}f_*\big).
\end{align*}

\end{remark}

\subsection{Comparison with the existing Landau gradient-flow structure}\label{sec:gf}
We recall the gradient-flow structure of the Landau equation \eqref{Landau} developed in \cite{carrillo2024landau}. At the formal level, the Landau equation \eqref{Landau} can be written as
\begin{align*}
\d_t f=-K^f_\Pi\dd\cH(f).
\end{align*}
Here, $K^f_\Pi$ is the Onsager operator associated with the Landau gradient $\tn_\Pi$. For every smooth and sufficiently decaying function $g$, it is given by
\begin{align}
\label{ons-Pi}
K^f_\Pi g=-\frac12 \tn_\Pi\cdot\big(ff_*\tn_\Pi g\big).
\end{align}
The Onsager operator induces a formal Riemannian structure through the Landau distance between probability densities $f_0$ and $f_1$, defined by
\begin{align*}
d_\Pi(f_0,f_1)^2
=\inf\Big\{
\frac12\int_0^1\int_{\{ff_*>0\}}
\frac{|U_\Pi|^2}{ff_*}\dd v_*\dd v\dd t
\Big\}.
\end{align*}
Here, the infimum is taken over all admissible pairs $(f,U_\Pi)$, whose precise definition can be found in Definition \ref{def:gre}, with $f:[0,1]\to\cP(\R^3)$ and $U_\Pi:[0,1]\times \R^6\to \R^3$, satisfying the grazing rate equation
\begin{align*}
\d_t f+\frac12\tn_\Pi\cdot U_\Pi=0,
\qquad
f|_{t=i}=f_i,\quad i=0,1.
\end{align*}
Several properties of the metric $d_\Pi(\cdot,\cdot)$ have been discussed in \cite{carrillo2024landau}, together with a variational characterisation theorem showing the gradient-flow structure via curves of maximal slope.


The cross-product gradient $\tn$ provides a different representation of the Landau gradient while inducing the same gradient-flow geometry as the gradient $\tn_\Pi$ used in \cite{carrillo2024landau}.
To compare the two formulations at a formal level, we consider a general non-negative symmetric kernel $A_0$ and the two Landau gradients
\begin{gather*}
\tn f=(v-v_*)\times(\nabla f-\nabla_* f_*),\quad \text{and}\\
\tn_{\Pi}f=\sqrt{A_0}|v-v_*|\Pi_{(v-v_*)^\perp}(\nabla f-\nabla_* f_*).
\end{gather*}
We have the following properties:
\begin{itemize}

\item (Symmetry).
The cross-Landau gradient $\tn f$ is symmetric under the exchange $v\leftrightarrow v_*$, whereas $\tn_{\Pi}f$ is antisymmetric:
\begin{align*}
(\tn f)(v_*,v)=(\tn f)(v,v_*),\qquad
(\tn_{\Pi}f)(v_*,v)=-(\tn_{\Pi}f)(v,v_*).
\end{align*}

\item (Onsager operator). 
Analously to \eqref{ons-Pi}, we define the Onsager operator  $K^f$ associated with the cross-product gradient $\tn$
\begin{align*}
K^f g=-\frac12 \tn\cdot\big(A_0ff_*\tn g\big).
\end{align*}
Using \eqref{cross} and a direct calculation, $\tn$ and $\tn_\Pi$ lead to the same Onsager operator, with the factor $A_0$ distributed differently between the kernel and the Landau gradient:
\begin{equation*}
\label{same-1}
\begin{aligned}
&K^f g=-\frac12 \tn\cdot\big(A_0ff_*\tn g\big)\\
=&{}\nabla_v\cdot\int_{v_*}A_0ff_*(v-v_*)
\times\big((v-v_*)\times(\nabla_vg-\nabla_{v_*}g_*)\big)\dd v_*\\
=&{}-\nabla_v\cdot\int_{v_*}A_0ff_*|v-v_*|^2
\Pi_{(v-v_*)^\perp}(\nabla_vg-\nabla_{v_*}g_*)\dd v_*\\
=&{}-\frac12\tn_\Pi\cdot\big(ff_*\tn_\Pi g\big)=K^f_\Pi g
\end{aligned}
\end{equation*}
for any sufficiently smooth and decaying $g:\R^3\to\R$. Here,
\begin{align*}
\tn_{\Pi}\cdot G
=\nabla_v\cdot\int_{\R^3}\sqrt{A_0}|v-v_*|
\Pi_{(v-v_*)^\perp}
\big(G(v,v_*)-G(v_*,v)\big)\dd v_*.
\end{align*}

\item (Grazing rate equations and curve actions).
The grazing rate equations and curve actions associated with $\tn$ and $\tn_\Pi$, respectively, are
\begin{gather*}
\d_t f+\frac12\tn\cdot U=0,
\qquad
\cA(f,U)=\frac12\int_{\R^6}\frac{|U|^2}{A_0ff_*},\\
\d_t f+\frac12\tn_\Pi\cdot U_\Pi=0,
\qquad
\cA_\Pi(f,U_\Pi)=\frac12\int_{\R^6}\frac{|U_\Pi|^2}{ff_*}.
\end{gather*}
When $f$ is an $\cH$-solution to the Landau equation \eqref{Landau}, the two curve actions are equal to the entropy dissipation:
\begin{align*}
\cA(f,U)=\cA_\Pi(f,U_\Pi)=\cD(f),
\end{align*}
where
\begin{align*}
U=-A_0ff_*\tn\log f,
\qquad
U_\Pi=-ff_*\tn_\Pi\log f.
\end{align*}

\item (Equivalent Landau metrics).
Using \eqref{cross}, the corresponding quadratic forms are the same
\begin{equation}
\label{eq:qua}
\begin{aligned}
A_0\big|\tn\varphi\big|^2
&=A_0\big|(v-v_*)\times
(\nabla\varphi-\nabla_*\varphi_*)\big|^2\\
&=A_0|v-v_*|^2
\big|\Pi_{(v-v_*)^\perp}
(\nabla\varphi-\nabla_*\varphi_*)\big|^2\\
&=\big|\tn_\Pi\varphi\big|^2.
\end{aligned}
\end{equation}

By \cite[Proposition 27]{carrillo2024landau}, the minimisation of $\cA_\Pi(f,U_\Pi)$ over all admissible grazing rates is attained by an optimal grazing rate $\widetilde U_\Pi$, namely,
\begin{align*}
\inf_{U_\Pi}\cA_\Pi(f,U_\Pi)
=\cA_\Pi(f,\widetilde U_\Pi),\quad \d_t f+\frac12 \tn_\Pi\cdot U_\Pi=0,
\end{align*}
where 
\begin{align*}
\widetilde U_\Pi=ff_*M_\Pi,
\quad
M_\Pi\in T^f_\Pi:=
\overline{
\big\{\tn_\Pi\varphi\mid\varphi\in C_c^\infty(\R^3;\R)}
\big\}^{L^2(ff_*\dd v\dd v_*)}.
\end{align*}

Analogously, in the cross-product formulation, the minimisation of $\cA(f,U)$ is attained by an optimal grazing rate $\widetilde U$ of the form
\begin{align*}
\widetilde U=A_0ff_*M,
\quad
M\in T^f:=
\overline{
\big\{\tn\varphi\mid\varphi\in C_c^\infty(\R^3;\R)\big\}
}^{L^2(A_0ff_*\dd v\dd v_*)},
\end{align*}
such that 
\begin{align*}
\inf_U\cA(f,U)
=\cA(f,\widetilde U),
\qquad
\d_t f+\frac12\tn\cdot U=0.
\end{align*}
The metric induced by the cross-product representation is given by
\begin{align*}
d(f_0,f_1)^2
=\inf_{\substack{\d_t f+\frac12\tn\cdot U=0\\
f|{t=0}=f_0,\,f|{t=1}=f_1}}
\Big\{
\frac12\int_0^1\int_{\{ff_*>0\}}
\frac{|U|^2}{A_0ff_*}\dd v_*\dd v\dd t
\Big\}.
\end{align*}
Together with the isometric identification $\frac{v-v_*}{|v-v_*|}\times$ between $T^f_\Pi$ and $T^f$,
the quadratic identity \eqref{eq:qua} shows that the metric induced by $\tn$ agrees with that induced by $\tn_\Pi$.

\end{itemize}

\begin{remark}[GENERIC approach to a fuzzy Landau equation]
In a series of works \cite{DH25a,DH25b,DGH25}, we study the GENERIC structure of the so-called fuzzy Landau equation, which serves as an approximation to the spatially inhomogeneous Landau equation. The model retains the transport mechanism but replaces spatially local collisions with delocalised interactions, thereby relaxing the spatial locality of the quadratic collision term:
\begin{align*}
&\d_t f+v\cdot\nabla_xf=Q_{\sf fuz}(f,f),\\
&Q_{\sf fuz}(f,f)
=\nabla_v\cdot
\int_{\R^{2d}}\kappa(x-x_*)A(|v-v_*|)\Pi_{(v-v_*)^\perp}
\big(f_*\nabla_vf-f\nabla_{v_*}f_*\big)
\dd x_*\dd v_*.
\end{align*}
Here, $x$ denotes position while $v$ denotes velocity. We write $f_*=f(x_*,v_*)$. The kernel $\kappa(x-x_*)$ describes the spatial distribution of the delocalised collisions. When $\kappa=\delta_0$, we (formally) recover the spatially inhomogeneous Landau equation.

The GENERIC \textit{(General Equation for Non-Equilibrium Reversible-Irreversible Coupling)} framework provides a systematic description of thermodynamically consistent systems combining Hamiltonian, or reversible, dynamics with dissipative gradient flows. For the fuzzy Landau equation, the transport term $v\cdot\nabla_x f$ plays the role of the Hamiltonian component, while the collision operator $Q_{\sf fuz}(f,f)$ plays the role of the gradient-flow component. In \cite{DH25a}, we establish a variational characterisation of the GENERIC structure of the fuzzy Landau equation, analogous to \eqref{J}.

The cross-Landau gradient can also be extended to the fuzzy setting by defining
\begin{align*}
\tn f
=(v-v_*)\times
\big(\nabla_vf(x,v)-\nabla_{v_*}f(x_*,v_*)\big).
\end{align*}
In \cite{DH25a}, we treat the collision term using the method developed for the spatially homogeneous case in \cite{carrillo2024landau}. As explained in Section~\ref{sec:vc}, replacing the Landau gradient changes only the argument for the collision term. Combining the new collision argument with the remaining arguments of \cite{DH25a}, one obtains the corresponding variational characterisation of the fuzzy Landau equation under only the weighted $L^1$ assumption
\begin{align*}
\big(\langle x\rangle^2+\langle v\rangle^{2+\gamma_+}\big)f
\in L^\infty\big([0,T];L^1(\R^{2d})\big),
\end{align*}
for kernels $A_0$ satisfying \eqref{A:T3}. In particular, no additional $L^p$ bounds are required.

Existence results for the fuzzy Landau equation are provided in \cite{DH25b}. Grazing limits based on the variational characterisation were established for the spatially homogeneous case in \cite{carrillo2022boltzmann} and for the fuzzy case in \cite{DGH25}. Since the quadratic forms induced by $\tn$ and $\tn_\Pi$ coincide, as observed in Section \ref{sec:gf}, the corresponding arguments remain applicable to the cross-Landau gradient $\tn$ and the kernels \eqref{A:T3}.

Several Landau gradients can be obtained by distributing the factors $A_0$ and $\Pi_{(v-v_*)^\perp}$ differently between the gradient and the kernel. The key is to choose the representation that is most convenient for the analysis. We refer the reader to \cite[Section 4]{DH25a} for further discussion. We emphasise that only the cross-product Landau gradient improves the $L^p$-bound requirements. 

\end{remark}

\subsection*{Acknowledgements}
M. H. D is funded by an EPSRC Standard Grant EP/Y008561/1. Z.~H. is funded by the Deutsche Forschungsgemeinschaft (DFG, German Research Foundation) – Project-ID 317210226 – SFB 1283.

\section{Commutation properties }\label{sec:com}


The key ingredient in the proof of Theorem~\ref{thm:main} is the chain rule in Lemma~\ref{lem:chain}, which is proved in Section~\ref{sec:vc}. Following the strategy of \cite{carrillo2024landau,DH25a,erbar2023gradient,EH25}, we regularise the curves by convolution with the family of Gaussian mollifiers
\begin{equation}
\label{sec2:def:M}
M_\beta(z)=(2\pi\beta)^{-3/2}\exp\left(-\frac{|z|^2}{2\beta}\right),
\qquad z\in\R^3,\quad \beta\in(0,1).
\end{equation}
After proving the chain rule for the regularised curves, we pass to the limit as $\beta\to0$.

The Landau gradient $ \tn_{\Pi} f=|v-v_*|^{\frac{\gamma}{2}+1}\Pi_{(v-v_*)^\perp}(\nabla f-\nabla_* f_*)$ used in \cite{carrillo2024landau}
does not commute with Gaussian regularisation. This lack of commutation leads to the additional $L^p$-assumptions in that approach.

Here, we split the difficulty into two parts: the radial kernel $A_0$ and the projection $\Pi_{(v-v_*)^\perp}$. We handle the projection using the cross-Landau gradient $\tn$. 

We use the notations $M_\beta*_{(v,v_*)}=M_\beta*_vM_\beta*_{v_*}$ and $M_\beta(v,v_*)=M_\beta(v)M_\beta(v_*)$.

The key technical inputs are the exact commutation properties 
    \begin{gather}
  \big[\sqrt{\Pi_{(v-v_*)^\perp}}(\nabla_v-\nabla_{v_*}),M_\beta*_{(v,v_*)}\big]=0 \quad \text{and} \label{com:both-1}\\ 
  \big[A_0^{\pm1}(|v-v_*|),M_\beta*_{(v,v_*)}\big]\lesssim A_0^{\pm1}(|v-v_*|). \label{com:both-2}
\end{gather}

Here, with a slight abuse of notation concerning $\sqrt{\Pi}$ and $\tn$, we refer to Lemma~\ref{lem:com-1} for the precise statement of the first commutator.

\begin{itemize}
\item To handle the commutator with $\sqrt{\Pi}(\nabla-\nabla_{*})$, we introduce the new Landau gradient and use the map
$$b\mapsto \frac{a}{|a|}\times b\qquad a,b\in\R^3
$$
as a square-root factorisation of $\Pi_{(v-v_*)^\perp}$, as in the definition \eqref{def:nabla-L}. We remark that this observation was used in \cite{Des16} to show that a weighted Fisher information is controlled by the entropy dissipation:
\begin{align*}
|a|^2b\cdot\big(\Pi_{a^\perp}b\big)=|a\times b|^2.
\end{align*}

\item To handle the commutator with the kernel $A_0$, we impose assumption \eqref{A:T3}
\begin{align*}
  A_0(|v-v_*|)\sim \langle v-v_*\rangle^{\gamma} \qquad \gamma\in(-\infty,1].
\end{align*}
This assumption and the corresponding commutator properties are also used in the variational characterisation of the gradient-flow structure of the Boltzmann equation \cite{erbar2023gradient,EH25}. However, this method does not apply to the kernels $A_0(|v-v_*|)=|v-v_*|^\gamma$ with $\gamma\in(-4,1]$, as considered in \cite{carrillo2024landau}, since the commutator bound \eqref{com:both-2} doe not generally hold for $|v-v_*|^\gamma$.
\end{itemize}

Concerning the commutator estimate \eqref{com:both-2}, we use the following lemma.
   \begin{lemma}[\cite{erbar2023gradient}, Lemma 2.4; \cite{EH25}, Lemma 2.1]\label{lem:kernel-convolve}
   Let $A_0$ satisfy \eqref{A:T3}
\begin{equation*}
 A_0(|v-v_*|)\sim \langle v-v_*\rangle^{\gamma}\quad \text{with}\quad\gamma\in(-\infty,1].
 \end{equation*}
There exists a constant $C$, independent of  $\beta\in(0,1)$, such that
\begin{align*}
A_0(|v-v_*|)^{\pm1}*_{(v,v_*)} M_\beta\leq CA_0(|v-v_*|)^{\pm1}.
\end{align*}
\end{lemma}

Concerning the vanishing commutator \eqref{com:both-1}, we extend the definition of $\tn$ to a lift setting $\af=\af(v,v_*)$.

 We define the mass-centred coordinate as follows 
\begin{align*}
 z_0=\frac{v+v_*}{2}\quad\text{and}\quad      z=\frac{v-v_*}{2},
\end{align*}
where $ z$ can be written in spherical coordinates as follows
\begin{equation*}
 z=r\sigma,\quad  r:=|z| \quad\text{and}\quad \sigma:=\frac{z}{|z|}\in \sd.
\end{equation*}
Notice that $\nabla_{ z }=\nabla_v-\nabla_{v_*}$ and 
\begin{gather*}
\nabla_{ z }=\sigma\d_r+\frac{1}{r}\nabla_\sigma ,\quad \nabla_\sigma :=| z |\Pi_{ z ^\perp}\nabla_{ z}.
\end{gather*}
For $\af:\R^6\to \R$, we abuse the notation 
\begin{align*}
\af(v,v_*)=\af(z_0, z )=\af(z_0,\ar,\sigma).
\end{align*}

We extend the definition of Landau gradient \eqref{def:nabla-L} as follows
\begin{align*}
    \TN\af&=(v-v_*)\times \nabla_{(v-v_*)^\perp}(\nabla_v-\nabla_{v_*})\af\\
    &=2\sigma\times \nabla_{\sigma}\af.
\end{align*}
Via the integration by parts formula
\begin{align*}
&\int_{\R^6} \af\TN\cdot \aU=-\int_{\R^6}\TN \af\cdot \aU,
\end{align*}
where $\aU:\R^6\to\R^3$, the lifted Landau cross-divergence is given by
\begin{align*}
    \TN\cdot \aU
    &= (\nabla_v-\nabla_{v_*})\cdot \Pi_{(v-v_*)^\perp}\big((v-v_*)\times \aU\big)\\
    &= 2\nabla_{\sigma}\cdot \big(\sigma\times \aU\big).
\end{align*}

For example, $\af$ can be taken as the tensorised density $\af=f\otimes f_*$. The tensorised representation and the spherical-gradient structure were recently used in \cite{GS25} to prove that the Fisher information is non-increasing along solutions of the Landau equation.

We show the following commutation lemma.
\begin{lemma}\label{lem:com-1}
Let $M_\beta$ be given as \eqref{sec2:def:M}. Let $\af\in L^1(\R^{6};\R)$ and $\aU\in L^1(\R^6;\R^3)$. For all $\beta\in(0,1)$, we have 
\begin{align}
   \label{commute}
   \big[\TN, M_\beta*_{(v,v_*)}\big]\af=0,\quad     \big[\TN\cdot , M_\beta*_{(v,v_*)}\big]\aU=0
   \end{align} 
   in a distribution sense.
\end{lemma}
\begin{proof}
 We define the mass-centred coordinate as follows 
 \begin{align*}
&\bar z_0=\frac{u+u_*}{2},\quad \bar z=  \frac{u-u_*}{2}=\bar r\bar \sigma,\\
&z_0=\frac{v+v_*}{2},\quad z=  \frac{v-v_*}{2}=r\sigma,
\end{align*}
where 
\begin{equation*}
 \bar r:=|\bar z|,\quad r:=|z| \quad\text{and}\quad \bar \sigma:=\frac{\bar z}{|\bar z|},\,\sigma:=\frac{z}{|z|}\in \sd.
\end{equation*}

Concerning the Gaussian mollifier 
$$M_\beta(v-u,v_*-u_*)=(2\pi\beta)^{-d}\exp\Big(-\frac{|v-u|^2+|v_*-u_*|^2}{2\beta}\Big),$$
by using the following identities
\begin{gather*}
v-u=(z_0-\bar z_0)+(z-\bar z),\quad v_*-u_*=(z_0-\bar z_0)-(z-\bar z),\\
|v-u|^2+|v_*-u_*|^2=|z_0-\bar z_0|^2+|z-\bar z|^2,\\
|z-\bar z|^2= |r\sigma-\bar r\bar \sigma|^2=r^2+\bar r^2-2r\bar r\sigma\cdot \bar \sigma,
\end{gather*}
we have 
\begin{equation}
    \label{M-beta}
\begin{aligned}
  &M_\beta(v-u,v_*-u_*)=M_\beta (z_0-\bar z_0,z-\bar z)\\
  =&{}(2\pi\beta)^{-d}\exp\Big(-\frac{|z_0-\bar z_0|^2+\bar r^2+r^2-2r\bar r\sigma\cdot \bar \sigma}{2\beta}\Big). 
\end{aligned}
\end{equation}

We note that, to prove \eqref{commute} in the sense of distributions, it suffices to establish it for sufficiently smooth and decaying $\af$ and $\aU$ by an integration-by-parts argument.

We first show $\big[\TN, M_\beta*_{(v,v_*)}\big]\af=0$. By changing of variables, we have 
\begin{align*}
   &\big(\TN \af\big)*_{(v,v_*)} M_\beta=2\int_{\bar z_0,\bar r,\bar \sigma}M_\beta (z_0-\bar z_0,z-\bar z)\bar\sigma\times \nabla_{\bar\sigma}\af(\bar z_0,\bar r,\bar\sigma)\bar r^{d-1}.
\end{align*}

By \eqref{M-beta}, the factor
$$
    \exp\Big(-\frac{|z_0-\bar z_0|^2+\bar r^2+r^2}{2\beta}\Big)\bar r^{d-1}
$$
commutes with $\bar\sigma\times\nabla_{\bar\sigma}$. Therefore, to prove that the commutator vanishes, it suffices to show that
\begin{equation}
\label{goal-1}
\int_{\bar\sigma}\exp\big(r\bar r\sigma\cdot \bar \sigma/\beta\big)\bar\sigma\times \nabla_{\bar \sigma}\af(\bar \sigma)=\sigma\times \nabla_\sigma\int_{\bar \sigma}\exp\big(r\bar r\sigma\cdot \bar \sigma/\beta\big) \af (\bar \sigma).
\end{equation}

We have the following key observation
\begin{align*}
\nabla_{\bar \sigma}(\bar \sigma\cdot\sigma)=\Pi_{\bar \sigma^\perp}\sigma'\quad \text{and}\quad    (\Pi_{\bar \sigma^\perp}\sigma)\times \bar \sigma=\sigma\times \bar\sigma,
\end{align*}
by integration by parts and the cross-product structure, we have 
\begin{align*}
&\int_{\bar \sigma}\exp\big(r\bar r\sigma\cdot \bar \sigma/\beta\big)\big( \bar \sigma\times \nabla_{\bar\sigma} \af (\bar\sigma)\big)\\
=&{}-\int_{\bar \sigma}\af(\bar\sigma)\bar \sigma\times \nabla_{\bar \sigma}\exp\big(r\bar r\sigma\cdot \bar\sigma/\beta\big)\\
=&{}-\int_{\bar \sigma}\af(\bar\sigma)\exp\big(r\bar r\sigma\cdot \bar\sigma/\beta\big)\frac{r\bar r}{\beta}(\bar \sigma\times \sigma)\\
=&{}\int_{\bar \sigma}\af(\bar\sigma)\sigma\times \nabla_{\sigma}\exp\big(r\bar r\sigma\cdot \bar\sigma/\beta\big)\\
=&{}\sigma\times \nabla_{\sigma}\int_{\bar \sigma}\af(\bar\sigma)\exp\big(r\bar r\sigma\cdot \bar\sigma/\beta\big).
\end{align*}
Hence, \eqref{goal-1} holds and the first commutator in \eqref{commute} holds.

The second commutator vanishes analogously. Indeed, the fact 
\begin{align*}
 (\Pi_{\bar \sigma^\perp}\sigma)\cdot(\bar \sigma\times \aU)=(\sigma\times \bar\sigma)\cdot \aU  \quad\text{and}\quad \nabla_{\sigma}\cdot(\sigma\times \bar\aU(\bar \sigma))=0
\end{align*}
imply that  
\begin{align*}
    &\int_{\bar\sigma}\exp(r\bar r\sigma\cdot \bar\sigma/\beta)\nabla_{\bar \sigma}\cdot \big(\bar\sigma\times \aU(\bar\sigma)\big)\\
     =&{}-\int_{\bar\sigma} \frac{r\bar r}{\beta}\exp(r\bar r\sigma\cdot \bar\sigma/\beta) (\Pi_{\bar\sigma^\perp}\sigma)\cdot\big(\bar\sigma\times \aU(\bar\sigma)\big)\\
       =&{}-\int_{\bar\sigma}\frac{r\bar r}{\beta}\exp(r\bar r\sigma\cdot \bar\sigma/\beta) (\sigma\times \bar\sigma)\cdot  \aU(\bar\sigma)\\
        =&{}\int_{\bar\sigma}\frac{r\bar r}{\beta}\exp(r\bar r\sigma\cdot \bar\sigma/\beta) \Pi_{\sigma}\bar\sigma\cdot \big(\sigma\times   \aU(\bar\sigma)\big)\\
        =&{}\int_{\bar\sigma}\nabla_{\sigma}\exp(r\bar r\sigma\cdot \bar\sigma/\beta)\cdot \big(\sigma\times   \aU(\bar\sigma)\big)\\
        =&{}\nabla_\sigma\cdot\Big(\sigma\times\int_{\bar\sigma}\exp(r\bar r\sigma\cdot \bar\sigma/\beta)  \aU(\bar\sigma)\Big).
\end{align*}

Hence, we have $\big[\TN\cdot , M_\beta*_{(v,v_*)}\big]\aU=0$.
    
\end{proof}

\section{The Chain rule}\label{sec:vc}

In this section, we prove Lemma~\ref{lem:chain}, which establishes the chain rule for the grazing rate equation
\begin{equation*}
\d_t f_t+\frac12\widetilde\nabla\cdot U_t=0,
\end{equation*}
namely,
\begin{align*}
\frac{d}{dt}\cH(f_t)
=\frac12\int_{\R^6}\widetilde\nabla\log f_t\cdot U_t\dd v_*\dd v.
\end{align*}

Then, applying the chain rule and a Cauchy--Schwarz argument as in \cite[Theorem 8]{carrillo2024landau} or \cite[Theorem 3.16]{DH25a}, we prove Theorem~\ref{thm:main}
\begin{align*}
\cH(f_T)&-\cH(f_0)
=\frac12\int_0^T\int_{\R^6}\widetilde\nabla\log f\cdot U\dd v_*\dd v\dd t\\
&\ge -\frac14\int_0^T\int_{\R^6}A_0ff_*\big|\tn\log f\big|^2\dd v_*\dd v\dd t
-\frac14\int_0^T\int_{\R^6}\frac{|U|^2}{A_0ff_*}\dd v_*\dd v\dd t\\
&=-\frac12\int_0^T\cD(f_t)\dd t-\frac12\int_0^T\cA(f_t,U_t)\dd t.
\end{align*}
Equality holds if and only if $U=-A_0ff_*\tn\log f$. Hence, $f$ is an $\cH$-solution of the Landau equation \eqref{Landau}.
In this case, the entropy identity \eqref{H-thm} follows by combining the entropy inequality
$$
    \cH(f_T)-\cH(f_0)+\int_0^T\cD(f_t)\dd t\le0
$$
with the variational formulation \eqref{J} $\cJ(f,U)\ge0$.

\subsection{Grazing rate equation}\label{TGE}

We define $\cP(\R^3)$ as the space of Borel probability measures on 
$\R^3$, which is endowed with the weak topology induced by duality with bounded continuous functions in $C_b(\R^3)$.

Next we recall the Boltzmann entropy functional for measure $\mu\in\cP(\R^3)$. If $\mu=f\cL$ is absolutely continuous with respect to Lebesgue measure $\cL$ with density $f$, we define
\begin{align*}
    \cH(\mu)=\int_{\R^3}f\log f\dd v.
\end{align*}
Otherwise, we define $\cH(\mu)=+\infty$. In the case of $\mu=f\cL$, we also write $\cH(f)$. 

We consider curves $(\mu_t)_{t\in[0,T]}\subset \cP(\R^3)$ derived by a grazing  rate continuity equation
\begin{equation}
\label{TGRE}
    \d_t \mu_t+\frac12 \widetilde \nabla \cdot \cU_t=0,
\end{equation}
where $\cU_t\in\cM(\R^6;\R^3)$ denotes the grazing rate, and $\cM(\R^6;\R^3)$ denotes the space of vector-valued signed Borel measures with bounded total variation on $\G$. The space $\cM(\R^6;\R^3)$ is endowed with the weak* topology induced by $C_0(\R^6;\R^3)$ with continuous functions vanishing at infinity.  

Analogously to \cite{carrillo2024landau,DH25a}, we define the notion of a weak solution to the grazing rate equation.
\begin{definition}\label{def:gre} We say $(\mu_t,\cU_t)$ is a pair of solutions to the transport grazing rate equation \eqref{TGRE} if 
\begin{enumerate}
    \item $\mu_t:[0,T]\to \cP(\R^3)$ is weakly continuous; 
    \item $(\cU_t)_{t\in[0,T]}$ is a family of Borel measures in $\M(\R^6;\R^3)$;
    \item $\int_0^T \dd|\cU_t|(\R^6)<\infty$;
    \item For any $\varphi\in C^\infty_c(\R^3)$ the following equality holds
    \begin{equation*}
        \frac{d}{dt}\int_{\R^3}\varphi\dd\mu_t=\frac12\int_{\R^6} \widetilde\nabla \varphi \dd \cU_t.
    \end{equation*}
    We define the set of such pairs $(\mu_t,\cU_t)$ by $\TGRE_T$. If the pair $(\mu_t,\cU_t)$ has  densities $f_t$ and $U_t$ with respect to Lebesgue measure on $\R^3$ and $\R^6$, we will also write $(f_t,U_t)\in\TGRE_T$. 
\end{enumerate}
\end{definition}

We define the function $\alpha:\R_+\times\R^3\to\R$ by letting
\begin{equation*}
\label{def:alpha}
 \alpha(s,u)=\left\{
 \begin{aligned}
& \frac{|u|^2}{2s},\quad s\neq0,\\
&0,\quad u=0,\,s=0,\\
&+\infty,\quad u\neq0,\,s=0.
 \end{aligned}
 \right.   
\end{equation*}
Notice that the function $(a,b)\mapsto \frac{|a|^2}{b}$ is joint convex on $\R\times\R_+$. The function $\alpha$ is lower semicontinuous, convex and 1-homogeneous, i.e. $\alpha(rs,ru)=r\alpha(s,u)$ for all $r\ge0$.

Next we define the action functional for $\mu\in\cP(\R^3)$ and $\cU\in \cM(\R^6;\R^3)$.
\begin{definition}[action functional]
\label{def: action}
We define the action functional
\begin{equation*}
\label{eq:action}
\cA(\mu,\cU)=\int_{\R^6}\alpha\Big(\frac{\dd \mu_1}{\dd \lambda},\frac{\dd \cU}{\dd \lambda}\Big)\dd\lambda,    
\end{equation*}
where 
\begin{align*}
&\mu_1(\dd v_* \dd v)\defeq A_0(|v-v_*|)\mu_t(\dd v)\mu_t(\dd v_*)\quad \text{and}\quad \lambda\defeq \mu_1+|\cU_t|.
\end{align*}
\end{definition}

We note that the curve action used in \cite{carrillo2024landau,DH25a} is defined by
\begin{align*}
     \cA_{\Pi}(f,U_{\Pi})=\frac12 \int_{\R^6}\frac{|U|^2}{ff_*}.
\end{align*}
When $f$ is an $\cH$-solution to the Landau equation \eqref{Landau}, the two curve actions agree with the entropy dissipation
\begin{align*}
  \cA_{\Pi}(f,U_{\Pi})=\cA(f,U)=\cD(f).  
\end{align*}
More precisely, $(f,U_{\Pi})$ and $(f,U)$ satisfy 
\begin{gather*}
    \d_t f=\tn_{\Pi}\cdot U_{\Pi}\quad \text{with}\quad U_{\Pi}=ff_*\tn_{\Pi} \log f,\\
     \d_t f=\tn\cdot U \quad \text{with}\quad U=A_0ff_*\tn\log f.
\end{gather*}

Arguing analogously to \cite[Section 3.2]{carrillo2024landau} or \cite[Section 3.1]{DH25a}, and taking the kernel $A_0(|v-v_*|)$ into account, we obtain the following properties of the curve action.
\begin{lemma}
\label{lem:U:density}
\begin{itemize}
\item (Density representation).
If $\mu=f\cL\in\cP(\R^3)$, $\cU\in\cM(\R^6;\R^3)$, and $\cA(\mu,\cU)<+\infty$, then there exists a Borel vector field $M:\R^6\to\R^3$ such that

$$
    \cU=Mff_*A_0(|v-v_*|)\dd v_*\dd v=U\dd v_*\dd v
$$

and
\begin{align*}
\cA(\mu,\cU)
=\frac12\int_{\R^6}|M|^2A_0ff_*\dd v_*\dd v
=\frac12\int_{\R^6}\frac{|U|^2}{A_0ff_*}\dd v_*\dd v.
\end{align*}

\item (Lower semicontinuity).
Let $\mu_n\rightharpoonup\mu$ in $\cP(\R^3)$ and $\cU_n\overset{*}{\rightharpoonup}\cU$ in $\cM(\R^6;\R^3)$. Then
\begin{align*}
\cA(\mu,\cU)\le\liminf_{n\to\infty}\cA(\mu_n,\cU_n).
\end{align*}

\item (Integrability estimates).
Let $(f_t,U_t)\in\TGRE_T$ and $\gamma\in(-\infty,1]$. If
\begin{align*}
C_A\defeq\int_0^T\cA(f_t,U_t)\dd t<+\infty,\quad 
C_E=\int_0^T\|f_t\|_{L^1_{2+\gamma_+}}\dd t<+\infty,
\end{align*}
then we have 
\begin{equation*}
\label{int:U:gamma:0}
\int_0^T\int_{\R^6}|U_t|
\big(\langle v\rangle^{1+\frac{\gamma_+}{2}}
+\langle v_*\rangle^{1+\frac{\gamma_+}{2}}\big)
\dd v_*\dd v\dd t
\le\sqrt{C_AC_E}.
\end{equation*}
\end{itemize}
\end{lemma}

\subsection{Chain rule}
We show the following chain rule.
\begin{lemma}[Chain rule]
\label{lem:chain}
Let $\gamma\in(-\infty,1]$. Let $(\mu_t,\cU_t)\in\TGRE_T$ such that $|\cH(\mu_0)|<+\infty$ and $(\mu_t)_{t\in[0,T]}\subset \cP_{2,2+\gamma_+}(\R^3)$. We assume 
        \begin{equation}
            \label{CR:AD}
            \int_0^T\sqrt{\cA(\mu_t,\cU_t)}\dd t<+\infty\quad\text{and}\quad  \int_0^T\sqrt{\cA(\mu_t,\cU_t)}\sqrt{\cD(\mu_t)}\dd t<+\infty.
        \end{equation}

 Let $\mu_t=f_t\cL$.
Then $|\cH(\mu_t)|<+\infty$ and the following chain rule holds
\begin{equation}
\label{int:chain-rule}
\cH(\mu_t)-\cH(\mu_s)=\frac12\int_s^t\int_{\R^6}\widetilde\nabla \log f_r\dd \cU_r\dd r
\end{equation}
for all $0\le s\le t\le T$. Moreover, the map $t\mapsto \cH(\mu_t)$ is absolutely continuous and we have
\begin{equation}
\label{ineq:chain-rule}
    \frac{d}{dt}\cH(\mu_t)=\frac12\int_{\R^6}\widetilde\nabla \log f_t\dd \cU_t
\end{equation}
for almost every $t\in[0,T]$.
\end{lemma}

The key step in proving Lemma~\ref{lem:chain} is to combine the commutation properties of the (tensorised) cross-Landau gradient from Lemma~\ref{lem:com-1} with the properties of the kernel $A_0(|v-v_*|)$ from Lemma~\ref{lem:kernel-convolve}. This yields the following proposition. 

\begin{proposition}\label{prop:potential-conv}
Let $\mu\in \cP(\R^6)$ and $\cU\in \cM(\Omega)$ be such that $\cA(\mu,\cU)<\infty$, $D(\mu)<\infty$, {  and $\mu$ has Lebesgue density $f$.} Then we have that
\begin{equation}\label{eq:potential-conv}
\begin{split}
&\lim_{\beta\to 0}\cA(\mu*_{v}M_\beta,\cU*_{(v,v_*)}M_\beta) =\cA(\mu,\cU)\;,\\
& \lim_{\beta\to 0}\cD(\mu*_{v}M_\beta) =\cD(\mu)\;.
\end{split}
\end{equation}
Moreover, there exists a constant $C>0$ such that for all $\beta\in(0,1)$ we have
\begin{equation}\label{eq:potential-bd}
\begin{split}
\cA(\mu*_{v}M_\beta,\cU*_{(v,v_*)}M_\beta) &\leq C \cA(\mu,\cU)\;,\\
\cD(\mu*_{v}M_\beta) &\leq C \cD(\mu)\;.
\end{split}
\end{equation}
\end{proposition}

\begin{proof}
We follow \cite[Lemma 4.2]{erbar2023gradient}, with an appropriate modification for the Landau case.

We use the notations in the tensorised setting. We recall the definition
\begin{align*}
    \af(v,v_*)=ff_*.
\end{align*}
By Lemma \ref{lem:U:density}, the functionals $\cD(\mu)$ and $\cA(\mu,\cU)$ can be written as
\begin{align*}
    \cD(\mu)=\cD(\af)=\frac12 \int_{\R^6}A_0\af |\TN \log \af|^2,\quad \cA(\mu,\cU)=\cA(\af,U)=\frac12 \int_{\R^6}\frac{|U|^2}{A_0\af}.
\end{align*}

We define $f^\beta=f*_vM_\beta$ and $\af^\beta=\af*_{(v,v_*)}M_\beta$. Notice that 
\begin{align*}
 \af^\beta=f^\beta f^\beta_*,\qquad \cD(\mu^\beta)=\cD(\af^\beta)\,\quad    \cA(\mu^\beta,\cU^\beta)=\cA(\af^\beta,U^\beta).
\end{align*}

We first show the claims of the functional $\cD$.  We note that
  \begin{align*}
    \cD(\af^\beta)=\int_{\R^6} A_0\af^\beta\big|\tn \log \af^\beta\big|^2=\int_{\R^6} A_0\frac{|\tn  \af^\beta|^2}{\af^\beta}.
  \end{align*}
 Note that $\frac{|\tn  \af^\beta|^2}{\af^\beta}$ converges point-wise to
  $\frac{|\tn  \af|^2}{\af}$ as $\beta\to0$. To show the convergence \eqref{eq:potential-conv}, we only need to show the uniform bound \eqref{eq:potential-bd} and apply the dominated convergence theorem. We then show \eqref{eq:potential-bd}. From the commutation Lemma \ref{lem:com-1} and the convexity of $(a,b)\mapsto \frac{|a|^2}{b}$ and Jensen's inequality we have 
\begin{align*}
        A_0\frac{|\tn  \af^\beta|^2}{\af^\beta}= A_0\frac{|(\tn  \af)^\beta|^2}{\af^\beta}\le A_0\Big(\frac{|\tn \af|^2}{\af}*M_\beta\Big).
      \end{align*}

    Moreover, by Lemma \ref{lem:kernel-convolve} that $A_0*M_\beta\leq C A_0$, we have 

    \begin{align*}
     \int_{\R^6}A_0\Big(\frac{|\tn \af|^2}{\af}*M_\beta\Big)= \int_{\R^6}\big(A_0*M_\beta\big)\frac{|\tn \af|^2}{\af}\lesssim   \int_{\R^6}A_0\frac{|\tn \af|^2}{\af}\lesssim \cD(\af).
    \end{align*}

The corresponding claims for the functional $\cA$ can be proven similarly by using convexity of the function $(a,b)\mapsto \frac{|a|^2}{b}$  as well as the bound $A_0^{-1}*M_\beta\leq C A_0^{-1}$.

\end{proof}

\subsection{Proof of Lemma \ref{lem:chain}}

We prove Lemma~\ref{lem:chain} in three steps: First, we regularise curves $(\mu_t,\cU_t)\in\TGRE_T$ so that they are smooth and decaying, with bounded entropy. We then prove the chain rule for the approximation curves. Finally, we pass to the limit to obtain \eqref{ineq:chain-rule}.

The proof strategy is analogous to the proof of the chain rule in \cite{carrillo2024landau,DH25a,erbar2023gradient,EH25}. Here, we sketch the standard arguments and detail only the passage to the limit in the velocity regularisation. This relies on the commutation properties of the cross-Landau gradient and Proposition \ref{prop:potential-conv}. 

We remark that in \cite{carrillo2024landau}, the mollifier $\exp(-\langle v\rangle)$ is used. However, convolution with this mollifier does not commute with the Landau gradient $\tn$. Hence, following \cite{DH25a,erbar2023gradient,EH25}, we use the Gaussian mollifier $\exp(-|v|^2)$. This requires the additional step of adding a decaying lower bound to ensure the integrability of $\tn\log f\cdot U$ by Lemma \ref{lem:U:density}.

\smallskip

The assumption \eqref{CR:AD} and Lemma \ref{lem:U:density} ensure the existence of densities $\mu=f\dd v$ and $\cU=U\dd v_*\dd v$.

We regularise $(\mu _t,\cU _t)$ in $v$ and $v,v_*$ by convoluting with $M_\beta(v)$, $\beta>0$, in the following way
    \begin{equation*}
        \mu^{\beta}_t=\mu _t*_v M_\beta \quad\text{and}\quad \cU^{\beta}_t=\cU _t*_{(v,v_*)}M_\beta.
    \end{equation*}

 Let $\eta\in C^\infty_c(\R;\R_+)$ such that
        $\supp(\eta)\subset[-1,1]$ and $\int_{\R}\eta=1$. We define the sequence of mollifiers $\eta^\delta=\delta^{-1}\eta(\frac{t}{\delta})$. We regularise $(\mu^{\alpha}_t,\cU^{\alpha }_t)$ in time in the following way
        \begin{equation*}
        \mu^{\beta ,\delta}_t=\int_{-\delta}^\delta\mu^{\beta}_{t-s}\eta^\delta_s\dd s\quad\text{and}\quad\cU^{\beta ,\delta}_t=\int_{-\delta}^\delta\cU^{\beta}_{t-s}\eta^\delta_s\dd s   
        \end{equation*}
        for $t\in I^\delta\defeq[\delta,T-\delta]$.
        
To ensure the boundedness of entropy, we add $\mu^{\beta,\delta}$ with the following lower bound
   \begin{equation*}
   \label{def:g}
      \mu^{\alpha ,\delta,\theta}_t=(1-\theta)\mu^{\beta ,\delta}_t+\theta g\cL,\quad  g(v)=\|\langle v\rangle^{-a}\|_{L^1}^{-1}\langle v\rangle^{-a}
   \end{equation*}
   for some large enough $a>0$. Correspondingly, we define 
   For $\theta\in(0,1)$, we define $\cU^{\alpha ,\delta,\theta}_t=(1-\theta)\cU^{\beta ,\delta}_t$.

    \smallskip 
    
Following \cite{carrillo2024landau,DH25a}, one can show that the regularised pair $(\mu^{\beta,\delta,\theta},\cU^{\beta,\delta,\theta})$ is smooth, has sufficient decay, and satisfies
\begin{equation*}
\begin{aligned}
\d_t \mu^{\beta,\delta,\theta}_t+\frac12\big(\tn\cdot \cU^{\delta,\theta}\big)^{\beta}=0.
\end{aligned}
\end{equation*}
In particular, applying the cross-Landau gradient, we have
\begin{equation}
\label{eq:tcre-err}
\begin{aligned}
(\tn\cdot \cU)^{\beta}=\tn\cdot \cU^{\beta}\quad\text{and}\quad
    \d_t \mu^{\beta,\delta,\theta}_t+\frac12\tn\cdot \cU^{\beta,\delta,\theta}=0.
\end{aligned}
\end{equation}

Indeed, by the normalisation of $M_\beta$ and Lemma~\ref{lem:com-1}, we have
\begin{align*}
   &\big( \tn\cdot \cU\big)*_v M_\beta\\
   =&{}M_\beta*_v\int_{v_*}\big(\nabla_v-\nabla_{v_*}\big)\cdot \Big(\Pi_{(v-v_*)^\perp}\big( (v-v_*)\times (U+U_*)\big)\Big)\\
   =&{} \int_{v_*}\int_{u,u_*}M_\beta(v-u)M_\beta(v_*-u_*)\big(\nabla_u-\nabla_{u_*}\big)\cdot \Big(\Pi_{(u-u_*)^\perp}\big( (u-u_*)\times (U+U_*)\big)\Big)\\
    =&{}\int_{v_*}M_\beta*_{(v,v_*)}\TN \cdot U=\int_{v_*}\TN \cdot U^\beta.
\end{align*}

The regularisation allows us to test \eqref{eq:tcre-err} by $\log f^{\beta,\delta,\theta}$ and obtain the following approximate chain rule:
\begin{equation*}
\begin{aligned}
\cH(f^{\beta,\delta,\theta}_t)-\cH(f^{\beta,\delta,\theta}_s)
=&\frac12\int_s^t \int_{\R^6} U^{\beta,\delta,\theta}_r\cdot \widetilde\nabla \log f^{\beta,\delta,\theta}_r\dd v_*\dd v\dd r.
\end{aligned}
\end{equation*}

The technical arguments for handling the time and lower-bound regularisations are analogous to those in \cite{EH25,erbar2023gradient,DH25b}. Hence, for presentational convenience, we omit the details concerning the regularisations in $\delta$ and $\theta$. We only highlight the difference arising from the regularisation in $v$. We first pass to the limit as $\delta\to0$ and then as $\theta\to0$ to arrive at
\begin{equation*}
\label{app:chain-rule-2}
\begin{aligned}
\cH(f^{\beta}_t)-\cH(f^{\beta}_s)
=&\frac12\int_s^t \int_{\R^6} U^{\beta}_r\cdot \widetilde\nabla \log f^{\beta}_r\dd v_*\dd v\dd r.
\end{aligned}
\end{equation*}

We show that
\begin{equation}
\label{U:conv}
   \lim_{\beta\to\infty} \int_s^t \int_{\R^6} U^{\beta}_r\cdot \widetilde\nabla \log f^{\beta}_r=\int_s^t \int_{\R^6} U_r\cdot \widetilde\nabla \log f_r
\end{equation}

By Cauchy--Schwarz inequality and Proposition \eqref{prop:potential-conv}, we have
\begin{align*}
  &\Big|\int_s^t \int_{\R^6}\widetilde\nabla \log f^{\beta}\cdot U^{\beta}\Big|\\
    \le&{}\int_0^T\sqrt{\int_{\R^6}A_0 f^{\beta}f_*^{\beta} |\tn \log  f^{\beta}f_*^{\beta}  |^2}\sqrt{\int_{\R^6}\frac{|U^\beta|^2}{A_0 f^{\beta}f_*^{\beta} }}\\
    =&{}4\int_0^T\sqrt{\cD(\mu^\beta)}\sqrt{\cA(\mu^\beta,\cU^\beta)}\le C \int_0^T\sqrt{\cD(\mu)}\sqrt{\cA(\mu,\cU)}<+\infty.
\end{align*}
Since $\tn\log f^{\beta}\cdot U^{\beta}\to\tn\log f\cdot U$ pointwise, the dominated convergence theorem yields \eqref{U:conv}.

 Since $\cH(f^{\beta}_t)$ is lower semicontinuous and  $f^{\beta}\to f$ in $L^1(\R^3)$ as $\beta\to0$, we have 
\begin{align*}
    \cH(f_t)\le \liminf_{\beta\to0} \cH(f^{\beta}_t).
\end{align*}
On the other hand, the convexity of $a\mapsto a\log a$ and Jensen's inequality imply for any $\beta\in(0,1)$
\begin{equation*}
\cH\Big(\int_{\R^3} f_t(v-u)M_\beta(u)\dd u\Big)\le\cH(f_t).
\end{equation*}
Hence, we conclude that
    $\lim_{\beta\to0} \cH(f^{\beta}_t)=\cH(f_t)$. We choose $s=0$ in the chain rule \eqref{int:chain-rule}, the boundedness of $\cH(\mu_0)$ and the right-hand side of the chain rule ensures that $\cH(f_t)$ is bounded and $t\mapsto \cH(f_t)$ is absolutely continuous.

    We prove the chain rule \eqref{ineq:chain-rule}.

\printbibliography

\end{document}